\documentclass[11pt, reqno]{amsart}
\usepackage{amsaddr}
\usepackage[latin1]{inputenc}
\usepackage{amsmath, amssymb}
\usepackage{centernot}
\usepackage{dsfont}
\usepackage{bbm}
\usepackage{float}
\usepackage[shortlabels]{enumitem}
\usepackage{graphicx}
\usepackage{amssymb}
\usepackage{amsfonts}
\RequirePackage{amsmath}
\RequirePackage{amsthm}

\usepackage{color}
\usepackage{dsfont}
\usepackage{geometry}
\usepackage{fancyhdr}
\usepackage{lipsum} 
\newtheorem{theorem}{Theorem}[section]
\newtheorem*{thm*}{Theorem}
\newtheorem{proposition}[theorem]{Proposition}
\newtheorem{lem}[theorem]{Lemma}

\usepackage{hyperref} 

\theoremstyle{definition}

\newtheorem{remark}{Remark}[section]

\newtheorem*{assumptions*}{Assumptions}

\DeclareMathOperator{\supp}{supp}

\providecommand{\abs}[1]{\left|#1\right|}
\providecommand{\norm}[1]{\left\lVert#1\right\rVert}

\renewcommand{\H}{\mathrm{H}}
\newcommand{\R}{\mathbb{R}}
\newcommand{\N}{\mathbb{N}}

\newcommand{\EE}{\mathbb{E}}

\renewcommand{\phi}{\varphi}
\renewcommand{\epsilon}{\varepsilon}

\newcommand{\X}{\mathcal{X}}

\title[]{A uniform rate of convergence for the entropic potentials in the quadratic Euclidean setting}
\author{Pablo L\'opez-Rivera}
\date{\today}
\address{Universit\'e Paris Cit\'e \& Sorbonne Universit\'e, CNRS, Laboratoire Jacques-Louis Lions (LJLL), F-75013 Paris, France}
\email{plopez@math.univ-paris-diderot.fr}
\calclayout
\begin{document}

\begin{abstract}
We bound the rate of uniform convergence in compact sets for both entropic potentials and their gradients towards the Brenier potential and its gradient, respectively. Both results hold in the quadratic Euclidean setting for absolutely continuous measures satisfying some convexity assumptions.
\end{abstract}

\maketitle

\section{Introduction}
\label{sec:section1}

\subsection{Optimal transport}
This note concerns the quadratic Euclidean optimal transport problem between absolutely continuous measures: given two probability measures $\mu$ and $\nu$ on $\R^d$ which are absolutely continuous with respect to the $d$-dimensional Lebesgue measure, we define the quadratic optimal transport problem associated to them, in its Kantorovich formulation, as
\begin{equation}
\label{eq:def_ot*}
\mathcal{C}_0(\mu, \nu):=\inf_{\pi \in \Pi(\mu, \nu)} \int_{\R^d\times \R^d}\frac12 \abs{x-y}^2\,d\pi(x,y),
\end{equation}
where $\Pi(\mu,\nu)$ denotes the set of probability measures on $\R^d\times \R^d$ with first and second marginals equal to $\mu$ and $\nu$, respectively. Among excellent monographs on general optimal transport theory, we can cite \cite{villani03, villani09,santambrogio15}. The Brenier-McCann theorem \cite{brenier91, mccann95} states that there exists a convex function $\phi_0\colon \R^d \to \R$ such that the map $T_0\colon \R^d \to \R^d$ defined by $T_0:= \nabla\phi_0$ pushes forward the measure $\mu$ towards $\nu$ and induces the unique optimal coupling $\pi_0 \in \Pi(\mu,\nu)$ for the variational problem \eqref{eq:def_ot*}. We say that $T_0$ is the Brenier or the optimal transport map and $\phi_0$ is a Brenier potential. Moreover, if we define $\psi_0:=\phi^*_0$ as the convex conjugate of $\phi_0$, the pair $(f_0,g_0):=\left(\frac12 \abs{\cdot}^2 - \phi_0, \frac12 \abs{\cdot}^2 - \psi_0\right)$ solves the dual problem to \eqref{eq:def_ot*}:
\begin{equation}
\label{eq:ot_dual*}
\mathcal{C}_0(\mu, \nu)=\sup_{\substack{(f,g) \in L^1(\mu)\times L^1(\nu), \\f\oplus g \leq \frac12 \abs{\cdot-\cdot}^2}} \int_{\R^d}f\, d\mu + \int_{\R^d}g \, d\nu = \int_{\R^d} f_0\,d\mu + \int_{\R^d} g_0\,d\nu,
\end{equation}
where for $(f,g) \in L^1(\mu)\times  L^1(\nu)$, the function $f \oplus g\colon \R^d \times \R^d \to \R$ denotes their sum, which is defined for $(x,y) \in \R^d \times \R^d$ by $(f\oplus g)(x, y) := f(x) + g(y)$.

\subsection{Entropic optimal transport}

We can regularize Problem \eqref{eq:def_ot*} by adding an entropy to the objective function: for $\epsilon>0$, we define the entropic regularization of the optimal transport problem (equivalent as well to the Schr\"odinger bridge \cite{leonard14}) as
\begin{equation}
\label{eq:def_eot*}
\mathcal{C}_\epsilon(\mu,\nu):=\inf_{\pi \in \Pi(\mu, \nu)} \int_{\R^d\times \R^d} \frac{1}{2} \abs{x-y}^2\,d\pi(x,y) + \epsilon\, \H(\pi | \mu\otimes \nu),
\end{equation}
where $\mathrm{H}(\cdot| \mu \otimes \nu)$ denotes the relative entropy functional with respect to the measure $\mu\otimes\nu$: for any $\eta \in \mathcal{P}(\R^d\times\R^d)$,
$$\mathrm{H}(\eta| \mu \otimes \nu) := \begin{cases}
    \int_{\R^d \times \R^d} \log \frac{d\eta}{d(\mu \otimes \nu)}\, d\eta, & \text{if } \eta \ll \mu \otimes \nu\\
    +\infty, & \text{if } \eta \centernot\ll \mu \otimes \nu.
\end{cases}$$
An excellent introductory reference to the subject is \cite{nutz2021introduction}.

Since the entropy functional $\mathrm{H}(\cdot| \mu \otimes \nu)$ is strictly convex, the problem \eqref{eq:def_eot*} has a unique solution that we denote by $\pi_\epsilon \in \Pi(\mu,\nu)$. On the other hand, there exists a pair of functions $(f_\epsilon,g_\epsilon) \in L^1(\mu)\times L^1(\nu)$, which is unique up to a constant, that solves the dual problem to \eqref{eq:def_eot*}:
\begin{equation}
\label{eq:eot_dual*}
\mathcal{C}_\epsilon(\mu, \nu)=\sup_{(f,g) \in L^1(\mu)\times L^1(\nu)} \int_{\R^d}f\, d\mu + \int_{\R^d}g \, d\nu - \epsilon \int_{\R^d \times \R^d} e^{\frac{f\oplus g - \frac12 \abs{\cdot-\cdot}^2}{\epsilon}} \, d(\mu \otimes \nu) + \epsilon = \int_{\R^d} f_\epsilon\,d\mu + \int_{\R^d} g_\epsilon\,d\nu.
\end{equation}
We define for each $\epsilon >0$ the entropic potentials as $(\phi_\epsilon, \psi_\epsilon) := \left(\frac12 \abs{\cdot}^2 - f_\epsilon, \frac12 \abs{\cdot}^2 - g_\epsilon\right)$, in analogy to their unregularized counterpart $(\phi_0, \psi_0)$.

\subsection{The connection between both problems}

It is known that computing $\phi_0$ or $T_0 = \nabla \phi_0$ is difficult, as one would solve the associated Monge-Amp\`ere equation \cite{figalli19}, a second-order nonlinear PDE. One of the advantages of the entropic problem is that the potentials $(\phi_\epsilon, \psi_\epsilon)$ are very tractable, numerically speaking, thanks to the Sinkhorn algorithm \cite{cuturi13, altschuler2017near, peyre2019computational}. On the other hand, as the regularization parameter $\epsilon$ vanishes, the problem \eqref{eq:def_eot*} converges in many senses to \eqref{eq:def_ot*}, a fact which allows us to approximate the optimal transport through the entropic regularization for small values of the parameter $\epsilon > 0$.

More precisely, as $\epsilon \to 0$, the entropic problem $\Gamma$-converges towards the unregularized one \cite{leonard12, carlier17}, which yields the convergence of the value functions, $\mathcal{C}_\epsilon(\mu,\nu) \to \mathcal{C}_0(\mu,\nu)$, and the convergence of the optimal couplings in the weak topology, $\pi_\epsilon \to \pi_0$. Concerning the dual optimizers, it is known that $\phi_\epsilon \to \phi_0$ (modulo subsequence) both in $L^1(\mu)$ and uniformly on compact sets \cite{gigli21, nutz22} and $\nabla \phi_\epsilon \to \nabla \phi_0$ (modulo subsequence) in $L^2(\mu)$ \cite{chiarini23}. All of the above results use $\Gamma$-convergence or compactness arguments, so a natural question is whether it is possible to quantify the rate at which these convergences happen.

In the case of the convergence of the value functions, several contributions have been made in the continuous setting: the first-order expansion $$\mathcal{C}_\epsilon(\mu,\nu) = \mathcal{C}_0(\mu,\nu) - \frac{d}2 \epsilon \log \epsilon + \frac{\epsilon}{2}\left(\H(\mu)+\H(\nu)\right) + o(\epsilon),$$
which holds as soon as $\mu$ has finite Fisher information, was proven to be true for the quadratic case and as a $\Gamma$-limit in dimension one in \cite{adams11,duong13} and in higher dimensions in \cite{erbar15}. The same expansion was established as a pointwise limit for strictly convex cost functions in \cite{pal24} (thus generalizing the quadratic case). The second-order expansion
$$\mathcal{C}_\epsilon(\mu,\nu) = \mathcal{C}_0(\mu,\nu) - \frac{d}2 \epsilon \log \epsilon + \frac{\epsilon}{2}\left(\H(\mu)+\H(\nu)\right) + \frac{\epsilon^2}{8}\mathrm{I}(\mu,\nu) + o(\epsilon^2),$$
where $\mathrm{I}(\mu,\nu)$ denotes the integrated Fisher information on the Wasserstein geodesic between $\mu$ and $\nu$, was proven for the Euclidean setting in \cite{chizat2020faster, conforti21}. For more general cost functions, the expansion
$$\mathcal{C}_\epsilon(\mu,\nu) = \mathcal{C}_0(\mu,\nu) - \frac{d}2 \epsilon \log \epsilon + O(\epsilon)$$
was established in \cite{carlier23,eckstein24}. In \cite{malamut2023convergenceratesregularizedoptimal}, the same formula was demonstrated. Additionally, the authors were able to identify the separate asymptotic behavior of $\H(\pi_\epsilon|\mu\otimes \nu)$ and $\mathcal{C}_\epsilon(\mu,\nu) - \H(\pi_\epsilon|\mu\otimes \nu)$ as $\epsilon$ vanishes.

Concerning the convergence of the optimal plans $\pi_\epsilon$, in \cite{malamut2023convergenceratesregularizedoptimal} the authors were able to quantify the 2-Wasserstein distance between $\pi_\epsilon$ and $\pi_0$ when the Brenier map $T_0$ is Lipschitz:
$$W_2(\pi_\epsilon, \pi_0) = \Theta(\sqrt{\epsilon}).$$

Up to now, the convergence of potentials in the continuous setting was only quantified in the $L^2(\mu)$ norm in terms of the difference of their gradients: in \cite{pooladian2021entropic}, for the quadratic setting under compactness and convexity assumptions, it was proven that
$$\norm{\nabla \phi_\epsilon - \nabla \phi_0}_{L^2(\mu)}^2 = O(\epsilon^2).$$
The rate
$$\norm{\nabla \phi_\epsilon - \nabla \phi_0}_{L^2(\mu)}^2 = O(-\epsilon \log \epsilon)$$
was found in \cite{carlier23} under slightly weaker assumptions. Finally, the rate
$$\norm{\nabla \phi_\epsilon - \nabla \phi_0}_{L^2(\mu)}^2 = O(\epsilon)$$
was established in \cite{malamut2023convergenceratesregularizedoptimal} for non-necessarily compactly supported measures.

This note aims to exhibit a bound for the rate of convergence of uniform convergence on compact sets of the potentials $\phi_\epsilon$ and their gradients $\nabla\phi_\epsilon$ as $\epsilon \to 0$, when both measures $\mu$ and $\nu$ are absolutely continuous with respect to the Lebesgue measure, under the quadratic cost. The following observation is the starting point of our analysis: the entropic potentials $\phi_\epsilon$ are convex functions (see Proposition \ref{prop:entropic_potentials_are_convex} below in Section \ref{sec:section2}). As mentioned above, they converge (up to a subsequence) uniformly on compact sets to $\phi_0$. It is a classical fact (see, for example, \cite[Theorem 25.7]{rockafellar70}) that convexity yields the convergence of the gradients $\nabla\phi_\epsilon$ to $\nabla\phi_0$ uniformly on compact sets as well. This justifies to search, for a fixed compact set $K\subseteq \R^d$, an asymptotic rate of convergence as $\epsilon$ goes to $0$ for $\norm{\nabla\phi_\epsilon - \nabla\phi_0}_{K,\infty}:=\sup_{x\in K} \abs{\nabla\phi_\epsilon(x)-\nabla\phi_0(x)}$.

If both $\mu$ and $\nu$ are Gaussian, we can explicitly bound $\norm{\nabla\phi_\epsilon - \nabla\phi_0}_{K,\infty}$. We present the result of this computation as an appetizer and defer its proof to the Section \ref{sec:section4}.

\begin{proposition}
\label{prop:gaussian_case}
Let $\mu=\mathcal{N}(0,A)$ and $\nu=\mathcal{N}(0,B)$ be two non-degenerate Gaussian measures with $A, B \in \mathcal{M}_d(\R)$ symmetric and positive-definite. For $R>0$, set $K:= \overline{\mathrm{B}}(0,R) \subseteq \R^d$ as the Euclidean closed ball of radius $R$. Then
\begin{equation}
\label{eq:prop_gaussian_case}
    \forall \epsilon > 0,\quad \norm{\nabla\phi_\epsilon - \nabla\phi_0}_{K,\infty}\leq R \epsilon \frac{\abs{A^{-1}}_{\mathrm{op}}}2 \left(\epsilon \frac{\abs{(A^{\frac12}BA^{\frac12})^{-1}}_\mathrm{op}^\frac{1}{2}}{4} + \epsilon^3\frac{\abs{(A^{\frac12}BA^{\frac12})^{-1}}_\mathrm{op}^\frac{3}{2}}{16} + 1\right),
\end{equation}
where $\abs{\cdot}_{\mathrm{op}}$ denotes the operator norm of a matrix.
\end{proposition}

Proposition \ref{prop:gaussian_case} reveals that we can quantify uniform convergence on compact sets as $O(\epsilon)$ as $\epsilon \to 0$, with a dimension-free bound depending only on the size of the compact set $K$ and the operator norms of the matrices $A$ and $B$.

\subsection{Main results}
\label{sec:main_results}

Our first result gives an estimate in the spirit of Proposition \ref{prop:gaussian_case} for absolutely continuous measures $\mu,\nu \in \mathcal{P}(\R^d)$  satisfying the following convexity assumptions (see the end of Section \ref{sec:section2} for a more involved discussion on those), in order to generalize what happens in the Gaussian case:

\begin{enumerate}[({A}1)]
    \item 
    The measures $\mu, \nu \in \mathcal{P}(\R^d)$ have the form $d\mu(x) = e^{-V(x)}\, dx$ and $d\nu(y) = e^{-W(y)}\,dy$, where $V, W\colon \R^d \to \R$ are smooth functions and there exist $\alpha, \beta > 0$ such that
    \begin{equation*}
    \forall x \in \R^d, \quad \nabla^2 V(x) \preccurlyeq \alpha I_d
    \end{equation*}
    and
    \begin{equation*}
    \forall y \in \R^d, \quad \nabla^2 W(y) \succcurlyeq \beta I_d,
    \end{equation*}
    where $I_d$ is the identity matrix of dimension $d$ and $\preccurlyeq$ denotes the  L\"owner order on the set of positive semidefinite matrices.
    \item 
    The measure $\mu$ satisfies a Poincar\'e inequality: there exists $C_{\mathrm{P}}(\mu) > 0$ such that for any $h\colon \R^d \to \R$ smooth with $\int_{\R^d}h\,d\mu = 0$,
    $$\norm{h}_{L^2(\mu)}^2 \leq C_{\mathrm{P}}(\mu) \norm{\nabla h}_{L^2(\mu)}^2.$$
    \item 
    The measure $\mu$ has finite differential entropy:
    $$-\infty < \H(\mu) := -\int_{\R^d} V(x) e^{-V(x)}\, dx < +\infty.$$
\end{enumerate}

\begin{theorem}
\label{thm:main_theorem_gradients}
Let $\mu$ and $\nu$ be two probability measures on $\R^d$ that are absolutely continuous with respect to the $d$-dimensional Lebesgue measure and satisfy the assumptions (A1), (A2), and (A3). Then, for any $K\subseteq \R^d$ compact, there exists a computable constant $C_{\mathrm{grad}} = C_{\mathrm{grad}}(d,K,\mu,\nu) >0$ such that
$$\forall \epsilon > 0,\quad \norm{\nabla\phi_\epsilon - \nabla\phi_0}_{K,\infty} \leq C_{\mathrm{grad}}\, \epsilon^{\frac{1}{d+4}}.$$
\end{theorem}

We remark that Theorem \ref{thm:main_theorem_gradients} does not provide an optimal bound; recall Proposition \ref{prop:gaussian_case}.

Our next result is a corollary of the previous theorem: we can also quantify the convergence in compact sets of the entropic potentials. For $K\subseteq \R^d$ compact we define, doing an abuse of notation, $\norm{\phi_\epsilon - \phi_0}_{K,\infty}:=\sup_{x\in K} \abs{\phi_\epsilon(x)-\phi_0(x)}$.

\begin{theorem}
\label{thm:main_theorem_potentials}
Let $\mu$ and $\nu$ be two probability measures on $\R^d$ that are absolutely continuous with respect to the $d$-dimensional Lebesgue measure and satisfy the assumptions (A1), (A2), and (A3). In addition, suppose that the following normalization holds: for every $\epsilon > 0$,
\begin{equation}
\label{eq:main_theorem_potentials_eq}
\int_{\R^d} \phi_\epsilon \, d\mu = \int_{\R^d} \phi_0 \, d\mu = 0.
\end{equation}
Then, for any $K\subseteq \R^d$ compact and connected, there exists a computable constant $C_{\mathrm{pot}} = C_{\mathrm{pot}}(d,K,\mu,\nu)>0$ such that
$$\forall \epsilon > 0,\quad \norm{\phi_\epsilon - \phi_0}_{K,\infty} \leq C_{\mathrm{pot}}\, \left(\epsilon^{\frac{1}{d+4}} + \epsilon\right).$$
\end{theorem}

To the best of our knowledge, this is the first work that addresses this problem for the entropic regularization in the continuous setting. Previously in the literature, this question has been answered both in the discrete \cite{cominetti94} and semi-discrete \cite{delalande2022nearly, sadhu2024approximationratesentropicmaps} settings for the entropic regularization. On the other hand, in \cite{gonzalezsanz2024}, a rate was found for the quadratic-regularized optimal transport problem in dimension one.

\subsection{Organization}
In Section \ref{sec:section2}, we will present all the necessary preliminaries on optimal transport and its entropic regularization, and detail the assumptions for our main results, Theorems \ref{thm:main_theorem_gradients} and \ref{thm:main_theorem_potentials}. Then, in Section \ref{sec:section3}, we will prove both results. Finally, in Section \ref{sec:section4}, we will prove Proposition \ref{prop:gaussian_case}.

\subsection{Acknowledgments}
I want to thank my advisor, Max Fathi, for his guidance, corrections, and continuous support with this project. I also thank Aymeric Baradat, Joaqu\'in Fontbona, Maxime Laborde, Flavien L\'eger, and Hugo Malamut for very useful conversations. I would like to thank Michael Goldman for pointing out a mistake in a previous version of this note. I am also very grateful to the referees and the editor for their comments, which have significantly improved the quality of this work.

This work has received support under the program ``Investissement d'Avenir'' launched by the French Government and implemented by ANR, with the reference ANR-18-IdEx-0001 as part of its program ``Emergence''. It was also supported by the Agence Nationale de la Recherche (ANR) Grant ANR-23-CE40-0003 (Project CONVIVIALITY). This project has received funding from the European Union's Horizon 2020 research and innovation programme under the Marie Sk\l{}odowska-Curie grant agreement No 945332.

\section{Preliminaries, notation and assumptions}
\label{sec:section2}

In this section, we state the basic notation used in this note and then review some properties of both the Brenier and entropic potentials that complement the basic statements previously introduced. Finally, we state our main assumptions and discuss their consequences and some sufficient conditions for them to hold.

\subsection{Basic notation}

For $x, y \in \R^d$, we denote the Euclidean inner product between $x$ and $y$ as $x\cdot y$ and the Euclidean norm of $x$ as $\abs{x}:=\sqrt{x\cdot x}$. We write the set of square matrices of dimension $d \times d$ with real entries as $\mathcal{M}_{d}(\R)$ and for $M \in \mathcal{M}_{d}(\R)$, we define its Hilbert-Schmidt norm, doing an abuse of notation, as $\abs{M}:= \sqrt{\sum_{i,j=1}^d \abs{M_{ij}}^2}$. For a twice-differentiable function $u \colon \R^d \to \R$ and $x \in \R^d$, we denote its gradient as $\nabla u(x) \in \R^d$ and its Hessian matrix as $\nabla^2 u(x) \in \mathcal{M}_{d}(\R)$. For a function $u \colon \R^d \to \R$, we define its convex conjugate as
$$\forall y \in \R^d, \quad u^*(y):= \sup_{x\in \R^d}(x\cdot y - \phi(x)).$$

For a Polish space $\mathcal{X}$ equipped with its Borel $\sigma$-algebra $\mathcal{B}(\X)$, we denote by $\mathcal{P}(\mathcal{X})$ the set of probability measures on $\mathcal{X}$; typically $\mathcal{X} \in \{\R^d, \R^d \times \R^d\}$. For a non-negative measure $\eta$ on $\mathcal{X}$, we denote its support by $\supp(\eta)$, and for $p \in [1, +\infty]$, we will write $L^{p}(\eta)$ for the classical $p$-Lebesgue space, which is a Banach space if we furnish it with the $p$-norm, $\norm{\cdot}_{L^p(\eta)}$. If $\mathcal{X}=\R^d$ and $\eta$ is the $d$-dimensional Lebesgue measure, we write the associated $p$-norm as $\norm{\cdot}_p$. If we have a sequence $(u_n)_{n \in \N}$ of continuous functions $u_n \colon \mathcal{X} \to \R$, we say that it converges uniformly on compact sets to a continuous $u\colon \mathcal{X} \to \R$ if for each $K \subseteq \mathcal{X}$ compact, $\norm{u_n - u}_{K,\infty} := \sup_{x \in K} \abs{u_n(x) - u(x)} \to 0$ as $n \to \infty$. If $u\colon \mathcal{X} \to \R^d$ or $u\colon \mathcal{X} \to \mathcal{M}_{d}(\R)$, we can extend the definitions of both $\norm{\cdot}_p$ and $\norm{\cdot}_{K,\infty}$ to them by doing the following abuses of notation: $\norm{u}_p := \norm{\abs{u}}_p$ and $\norm{u}_{K,\infty}:=\norm{\abs{u}}_{K,\infty}$, where $\abs{\cdot}$ denotes the Euclidean or Hilbert-Schmidt norm depending on the codomain of $u$.

Let $\mathcal{X}$ and $\mathcal{Y}$ be two Polish spaces. If we fix two probability measures $\mu \in \mathcal{P}(\mathcal{X})$ and $\nu \in \mathcal{P}(\mathcal{Y})$, we define the associated set of coupling measures $\Pi(\mu,\nu)$, also named the set of transport plans between $\mu$ and $\nu$, as the set containing all the measures $\pi \in \mathcal{P}(\mathcal{X} \times \mathcal{Y})$ such that for every $A \subseteq \mathcal{X}$ Borel, $\pi(A \times \mathcal{Y}) = \mu(A)$ and for every $B \subseteq \mathcal{Y}$ Borel, $\pi(\mathcal{X} \times B) = \nu(B)$. We immediately observe that $\mu \otimes \nu$, the product measure between $\mu$ and $\nu$, is in $\Pi(\mu, \nu)$, so $\Pi(\mu, \nu)$ is a nonempty set. Given a Borel map $T\colon \mathcal{X} \to \mathcal{Y}$, we say that it pushes forward the measure $\mu$ towards $\nu$ if for any $B \subseteq \mathcal{Y}$ Borel, $\nu(B) = \mu(T^{-1}(B))$. We remark that the map $T$ induces a natural coupling between $\mu$ and $\nu$, which is given by the pushforward of the measure $\mu$ by the map $x \mapsto (x, T(x))$.

\subsection{Further properties of the potentials}

We recall the definitions of the Brenier potentials $(\phi_0,\psi_0)$ and their entropic counterparts $(\phi_\epsilon,\psi_\epsilon)$. Note that they are not unique, since for any constant $\alpha \in \R$, then $(\phi_0 + \alpha, \psi_0 -\alpha)$ and $(\phi_\epsilon + \alpha, \psi_\epsilon -\alpha)$ are two new pairs of Brenier and entropic potentials, respectively. Therefore, if we need to enforce the uniqueness of the potentials, an additional normalization condition has to be imposed, for example, \eqref{eq:main_theorem_potentials_eq}. Nevertheless, we remark that their gradients are uniquely determined.

Recall that for $\epsilon > 0$ we have that $(f_\epsilon, g_\epsilon) = \left(\frac12\abs{\cdot} - \phi_\epsilon, \frac12\abs{\cdot} - \psi_\epsilon\right)$. The pair $(f_\epsilon, g_\epsilon)$ is intimately linked with $\pi_\epsilon$, the optimal coupling for \eqref{eq:def_eot*}, since
\begin{equation}
\label{eq:charact_pi_eps}
\frac{d\pi_\epsilon}{d(\mu\otimes\nu)}(x,y) = \exp\left(\frac{f_\epsilon(x) + g_\epsilon(y) - \frac12 \abs{x-y}^2}{\epsilon}\right) \quad \mu \otimes \nu\text{-a.s.}
\end{equation}

A direct consequence of \eqref{eq:charact_pi_eps}, is that the pair $(f_\epsilon, g_\epsilon)$ satisfies the so-called Schr\"odinger system: for any $(x,y) \in \supp(\mu) \times \supp(\nu)$,
\begin{equation}
\label{eq:schr_eq_f}
f_\epsilon (x) = -\epsilon \log\left(\int_{\R^d}e^{\frac1\epsilon \left[g_\epsilon(z) - \frac12 \abs{x-z}^2\right]}\, d\nu(z)\right),
\end{equation}
\begin{equation}
\label{eq:schr_eq_g}
g_\epsilon (y) = -\epsilon \log\left(\int_{\R^d}e^{\frac1\epsilon \left[f_\epsilon(z) - \frac12 \abs{y-z}^2\right]}\, d\mu(z)\right).
\end{equation}
The Schr\"odinger system allows us to extend $f_\epsilon$ and $g_\epsilon$, which are defined on $\supp(\mu)$ and $\supp(\nu)$, respectively, to all of $\R^d$ since the right-hand sides of equations \eqref{eq:schr_eq_f} and \eqref{eq:schr_eq_g} are well defined on all of $\R^d$. We also note that the regularity of the quadratic cost ensures that the potentials are continuous, hence real valued \cite[Lemma 3.1]{nutz22}. Moreover, the Schrödinger system can be exploited to derive finer properties of the pair $(\phi_\epsilon,\psi_\epsilon)$ such as convexity.

\begin{proposition}
\label{prop:entropic_potentials_are_convex}
Let $\mu, \nu \in \mathcal{P}(\R^d)$. Then, for any $\epsilon >0$, the associated potentials $\phi_\epsilon$ and $\psi_\epsilon$ are convex.
\end{proposition}

\begin{proof}
Let $\lambda \in (0,1)$ and $x_1, x_2 \in \R^d$. Then by \eqref{eq:schr_eq_f} we get
\begin{align*}
\lambda \phi_\epsilon(x_1) + (1-\lambda)\phi_\epsilon(x_2) &=  \epsilon\log\left(\left[\int_{\R^d}e^{\frac1\epsilon \left[x_1\cdot y - \psi_\epsilon(y)\right]}\, d\nu(y)\right]^\lambda \left[\int_{\R^d}e^{\frac1\epsilon \left[x_2\cdot y - \psi_\epsilon(y)\right]}\, d\nu(y)\right]^{1-\lambda}\right)\\
&\geq  \epsilon\log\left(\int_{\R^d}e^{\frac1\epsilon \left[(\lambda x_1 + (1-\lambda)x_2)\cdot y - \psi_\epsilon(y)\right]}\,d\nu(y)\right)\\
&=\phi_\epsilon(\lambda x_1 + (1-\lambda)x_2),
\end{align*}
where we used H\"older's inequality for $p = 1/\lambda$ and $q=1/(1-\lambda)$. The convexity of $\psi_\epsilon$ follows similarly from \eqref{eq:schr_eq_g}.
\end{proof}

\subsection{Assumptions}

Our main results are stated in terms of two probability measures $\mu, \nu \in \mathcal{P}(\R^d)$ that are absolutely continuous with respect to the $d$-dimensional Lebesgue measure and satisfy the following assumptions, which were already stated in Section \ref{sec:main_results}, but we repeat for convenience of the reader:
\begin{enumerate}[({A}1)]
    \item \label{as:a1} The measures $\mu, \nu \in \mathcal{P}(\R^d)$ have the form $d\mu(x) = e^{-V(x)}\, dx$ and $d\nu(y) = e^{-W(y)}\,dy$, where $V, W\colon \R^d \to \R$ are smooth functions and there exist $\alpha, \beta > 0$ such that
    \begin{equation}
    \label{eq:main_theorem_assumption_1}
    \forall x \in \R^d, \quad \nabla^2 V(x) \preccurlyeq \alpha I_d
    \end{equation}
    and
    \begin{equation}
    \label{eq:main_theorem_assumption_2}
    \forall y \in \R^d, \quad \nabla^2 W(y) \succcurlyeq \beta I_d,
    \end{equation}
    where $I_d$ is the identity matrix of dimension $d$ and $\preccurlyeq$ denotes the  L\"owner order on the set of positive semidefinite matrices.
    \item \label{as:a2} The measure $\mu$ satisfies a Poincar\'e inequality: there exists $C_{\mathrm{P}}(\mu) > 0$ such that for any $h\colon \R^d \to \R$ smooth with $\int_{\R^d}h\,d\mu = 0$,
    $$\norm{h}_{L^2(\mu)}^2 \leq C_{\mathrm{P}}(\mu) \norm{\nabla h}_{L^2(\mu)}^2.$$
    \item \label{as:a3} The measure $\mu$ has finite differential entropy:
    $$-\infty < \H(\mu) := -\int_{\R^d} V(x) e^{-V(x)}\, dx < +\infty.$$
\end{enumerate}

The main consequence of \ref{as:a1} is that we obtain quantitative control on both $\nabla^2\phi_0$ and $\nabla^2\phi_\epsilon$ for every $\epsilon >0$; that is, both $\nabla\phi_0$ and $\nabla\phi_\epsilon$ are Lipschitz and we have explicit control on the value of their Lipschitz constants. The following generalization of Caffarelli's contraction theorem \cite{caffarelli00} provides global Lipschitz regularity for the Brenier map $\nabla\phi_0$, which pushes $\mu$ towards $\nu$. The statement that we use here can be found on \cite[Theorem 1]{chewi2023entropic}, for example.

\begin{theorem}
\label{thm:caffarelli}
Suppose that both $\mu, \nu \in \mathcal{P}(\R^d)$ satisfy \ref{as:a1}. Then the Brenier map $\nabla \phi_0$ is globally Lipschitz. Moreover, the following estimate holds:
$$\forall x \in \R^d, \quad 0 \preccurlyeq \nabla^2 \phi_0(x) \preccurlyeq \sqrt{\frac{\alpha}{\beta}} I_d.$$
\end{theorem}

For the entropic counterpart of the Brenier map, namely $\nabla \phi_\epsilon$, we can also exhibit bounds on its derivative. These bounds were used in \cite{fathi20, chewi2023entropic} to give alternative proofs of Theorem \ref{thm:caffarelli} based on the entropic regularization. Here we use the ones proven in \cite[Theorem 4]{chewi2023entropic}.

\begin{theorem}
\label{thm:chewi-pooladian}
Suppose that both $\mu, \nu \in \mathcal{P}(\R^d)$ satisfy \ref{as:a1}. Then the entropic Brenier map $\nabla \phi_\epsilon$ is globally Lipschitz. Moreover, the following estimate holds:
$$\forall x \in \R^d, \quad 0 \preccurlyeq \nabla^2 \phi_\epsilon(x) \preccurlyeq \frac12 \left(\sqrt{\frac{4\alpha}{\beta} + \epsilon^2 \alpha^2}-\epsilon \alpha\right) I_d.$$
\end{theorem}

As we will see in the following remark, the functional inequality provided by \ref{as:a2} also entails regularity properties for the measure $\mu$ itself.

\begin{remark}
\label{rk:finite_moments}
A Poincar\'e inequality implies that the measure has finite moments of all orders \cite[Proposition 4.4.2]{bgl14}, so under \ref{as:a2}, the measure $\mu$ will have this property.
\end{remark}

\begin{remark}
\label{rk:log-concavity_nec_cond}
Log-concavity is a sufficient condition that entails both \ref{as:a2} and \ref{as:a3}. We say that the measure $d\mu(x) = e^{-V(x)}\, dx$ is log-concave if $V$ is convex. Since we assumed $V$ to be smooth, this is equivalent to
$$\forall x \in \R^d, \quad \nabla^2 V(x) \succcurlyeq 0.$$
First, log-concavity yields the validity of a Poincar\'e inequality (see, for example, \cite[Theorem 4.6.2]{bgl14}). Second, let us see that log-concavity implies \ref{as:a3}: indeed,
$$\H(\mu) = - \int_{\R^d} V(x) e^{-V(x)}\,dx = - \int_{\R^d} V(x) \, d\mu(x) \leq -V(0) - \nabla V(0)\cdot \EE_{X\sim \mu}[X] < + \infty,$$
where we used the convexity of $V$. For the lower bound, we recall that $\mu$ has finite moments of all orders (see Remark \ref{rk:finite_moments}), so
$$-\infty < \H(\mathcal{N}) \leq \H(\mu),$$
where $\mathcal{N}$ is the $d$-dimensional Gaussian measure with the same mean and covariance as $\mu$.
\end{remark}

\begin{remark}
\label{rk:also_nu}
Note that under \ref{as:a1}, the measure $\nu$ is log-concave (see \eqref{eq:main_theorem_assumption_2}); in particular, it has finite differential entropy and finite moments of all orders as well.
\end{remark}

Under the assumption \ref{as:a1}, we saw that Theorem \ref{thm:caffarelli} ensures the Lipschitz regularity of the Brenier map quantitatively. A direct consequence of this fact is the following: we can control the difference of the gradients of the potentials in the $L^2(\mu)$ norm by the difference between both costs. 

\begin{proposition}
Suppose that the Brenier map is $L$-Lipschitz. Then
$$\norm{\nabla\phi_\epsilon-\nabla\phi_0}_{L^2(\mu)}^2 \leq 2L \langle \pi_\epsilon-\pi_0, \frac12 \abs{\cdot - \cdot}^2 \rangle = 2L\left( \mathcal{C}_\epsilon(\mu,\nu) - \epsilon\H(\pi_\epsilon | \mu \otimes \nu) - \mathcal{C}_0(\mu,\nu)\right).$$
\end{proposition}

\begin{proof}
First, let us remark that a direct computation gives
\begin{equation}
\label{eq:gradient_eot_potential}
\forall x \in \R^d, \quad \nabla\phi_\epsilon(x) = \int_{\R^d} y \, d\pi_{\epsilon}^x(y),
\end{equation}
where for $x \in \R^d$, we have defined $d\pi_{\epsilon}^x(y) := \exp\left(\frac{f_\epsilon(x) + g_\epsilon(y) - \frac12 \abs{x-y}^2}{\epsilon}\right) \, d\nu(y)$. In other words, $(\pi_{\epsilon}^x)_{x \in \R^d}$ is a disintegration of $\pi_\epsilon$ with respect to $\mu$. Then, by Jensen's inequality, we see that
\begin{align*}
\norm{\nabla\phi_\epsilon-\nabla\phi_0}_{L^2(\mu)}^2 &= \int_{\R^d} \abs{\nabla\phi_\epsilon(x)-\nabla\phi_0(x)}^2 \, d\mu(x) = \int_{\R^d} \abs{\int_{\R^d}(y-\nabla\phi_0(x)) \, d\pi_{\epsilon}^x(y)}^2 \, d\mu(x)\\
&\leq \int_{\R^d}\int_{\R^d}  \abs{y-\nabla\phi_0(x)}^2 \, d\pi_{\epsilon}^x(y) \, d\mu(x) = \int_{\R^d\times \R^d} \abs{y-\nabla\phi_0(x)}^2 \, d\pi_{\epsilon}(x,y)\\
&= \norm{y-T_0(x)}_{L^2(\pi_\epsilon)}^2.
\end{align*}
Finally, we conclude by using the bound 
\[\norm{y-T_0(x)}_{L^2(\pi_\epsilon)}^2 \leq 2L \langle \pi_\epsilon-\pi_0, \frac12 \abs{\cdot - \cdot}^2 \rangle = 2L\left( \mathcal{C}_\epsilon(\mu,\nu) - \epsilon\H(\pi_\epsilon | \mu \otimes \nu) - \mathcal{C}_0(\mu,\nu)\right),\]
which was proven in both \cite[Lemma 3.8]{malamut2023convergenceratesregularizedoptimal} and \cite[Proposition 4.5]{carlier23} when the map $T_0$ is $L$-Lipschitz. Let us remark that this result, in turn, is based on an unpublished result of Ambrosio reported in \cite{gilgi11}.
\end{proof}

That is, if we are able to control $\langle \pi_\epsilon-\pi_0, \frac12 \abs{\cdot - \cdot}^2 \rangle$, then we can control the $L^2$ norm of the difference. In \cite[Theorem 3.7]{malamut2023convergenceratesregularizedoptimal}, we can find this control, which can be applied as a consequence of both \ref{as:a2} and \ref{as:a3}, recall Remarks \ref{rk:finite_moments} and \ref{rk:also_nu}.

\begin{theorem}{(\cite[Theorem 3.7]{malamut2023convergenceratesregularizedoptimal}).}
\label{thm:malamut-sylvestre}
Suppose that both $\mu$ and $\nu$ have finite moments of order $2 + \delta$, for some $\delta > 0$, and that both have finite differential entropy. Then there exists $C=C(d,\mu,\nu)>0$ depending on the dimension $d$ and on the moments of order $2 + \delta$ and the differential entropies of both $\mu$ and $\nu$ such that for every $\epsilon > 0$,
$$\langle \pi_\epsilon-\pi_0, \frac12 \abs{\cdot - \cdot}^2 \rangle \leq C \epsilon.$$
\end{theorem}

That is, under our three assumptions, there exists a constant $C_1 = C_1(d, \mu, \nu) >0$ depending on $d, \mu$ and $\nu$ such that
\begin{equation}
\label{eq:step_3_0}
\norm{\nabla\phi_\epsilon - \nabla\phi_0}^2_{L^2(\mu)} \leq C_1 \epsilon.
\end{equation}

Now let us assume that the following normalization holds:
\begin{equation}
\label{eq:norms_pots}
\forall \epsilon > 0, \quad \int_{\R^d} \phi_\epsilon \, d\mu = \int_{\R^d} \phi_0 \, d\mu = 0.
\end{equation}
Note that if we are working with the gradients of the potentials, they are uniquely determined, so we may assume, without loss of generality, that \eqref{eq:norms_pots} holds. In particular, under this normalization, we can use \ref{as:a2} to control the $L^2$ norm of the difference of the potentials:
\begin{equation}
\label{eq:extra_poincare}
\norm{\phi_\epsilon - \phi_0}^2_{L^2(\mu)} \leq C_{\mathrm{P}}(\mu) \norm{\nabla\phi_\epsilon - \nabla\phi_0}^2_{L^2(\mu)}.
\end{equation}
That is, again under all our assumptions, we have that there exists a constant $C_2 = C_2(d,\mu,\nu) >0$ depending on $d, \mu$ and $\nu$ such that
\begin{equation}
\label{eq:step_3_1}
\norm{\phi_\epsilon - \phi_0}^2_{L^2(\mu)} \leq C_2 \epsilon.
\end{equation}

\section{Proof of the main results}
\label{sec:section3}

This section aims to prove Theorem \ref{thm:main_theorem_gradients} and then to obtain as a corollary Theorem \ref{thm:main_theorem_potentials}. The starting point of our proof will be the Gagliardo-Nirenberg inequality, which allows us to control the $p$-norm of the derivatives of order $j$ of a function by the $r$-norm of its derivatives of order $k$ and the $q$-norm of the function itself, where the parameters $i, j, p, q$, and $r$ satisfy some relations. We state the inequality in the following version, found in \cite[Theorem 3.1]{soudsky18}, since it allows the critical values $p=+\infty$ and $r = +\infty$. Here the ambient space is $\R^d$ and we recall the notation $\norm{\cdot}_s = \norm{\cdot}_{L^s(\mathrm{Leb})}$.

\begin{theorem}[Gagliardo-Nirenberg]
\label{thm:gn_inequality}
Let $j,k \in \N$ such that $1 \leq j < k$. Let $\theta \in [j/k,1]$, $q \in [1,+\infty]$, $p \in (1, +\infty]$, and $r \in [1,+\infty]$ such that
\begin{equation}
\label{eq:gn_inequality_aux_1}
\frac1p = \frac{j}d + \theta \left(\frac1r - \frac{k}d\right) + \frac{1-\theta}{q}
\end{equation}
and
\begin{equation}
\label{eq:gn_inequality_aux_2}
\forall 0 \leq i \leq k-j-1, \quad r^{(i)} \neq d,
\end{equation}
where $r^{(0)} := r$ and $r^{(i)}:=\frac{dr^{(i-1)}}{d-r^{(i-1)}}$ for $i \geq 1$. Then there exists a constant $C_{\mathrm{GN}} = C_{\mathrm{GN}}(j,k,d,\theta,p,q,r)>0$ such that for any $u \colon \R^d \to \R$ sufficiently smooth and integrable,
\begin{equation}
\label{eq:gn_inequality}
\norm{\nabla^j u}_p \leq C_{\mathrm{GN}} \norm{\nabla^k u}_r^\theta \norm{u}_q^{1-\theta}.
\end{equation}
\end{theorem}

We are ready to start the proof of Theorem \ref{thm:main_theorem_gradients}.

\begin{proof}[Proof (of Theorem \ref{thm:main_theorem_gradients}):]

The proof is divided into three consecutive steps.

\textbf{Step 1:} The first step will be to choose appropriate parameters to apply the Gagliardo-Nirenberg inequality \eqref{eq:gn_inequality}. From Theorem \ref{thm:gn_inequality} and choosing
$$j=1, k=2, p =+\infty, q=2, r=+\infty,$$
the value of $\theta$ is determined by \eqref{eq:gn_inequality_aux_1}, so we get
$$\theta = \frac{d+2}{d+4} \in [1/2,1].$$
On the other hand, we notice that the condition \eqref{eq:gn_inequality_aux_2} is trivially verified since $r^{(0)} = +\infty$. Hence, there exists a constant $C_{\mathrm{GN}} = C_{\mathrm{GN}}(d) > 0$ depending on the dimension $d$ such that for any $u \colon \R^d \to \R$ sufficiently smooth and integrable,
\begin{equation}
\label{eq:gn_particular_parameters}
\norm{\nabla u}_{\infty} \leq C_{\mathrm{GN}} \norm{\nabla^2 u}_{\infty}^\frac{d+2}{d+4} \norm{u}_2^{\frac{2}{d+4}}.
\end{equation}
Now let $K \subseteq \R^d$ be a compact set, take $R>0$ such that $K \subseteq \mathrm{B}(0,R)$, where $\mathrm{B}(0,R)$ denotes the open ball with center $0$ and radius $R$, and such that $\mu(\mathrm{B}(0,R)) > 4/5$, and let $w \colon \R^d \to \R_+$ be a compactly supported smooth function satisfying the following properties:
\begin{itemize}
    \item $\forall x \in \R^d$, $0 \leq w(x) \leq 1$;
    \item $\forall x \in K$, $w(x) = 1$; and
    \item $\supp(w) \subseteq \mathrm{B}(0,R)$.
\end{itemize}
Let $\epsilon>0$. Then
\[\norm{\nabla\phi_\epsilon-\nabla\phi_0}_{K,\infty} \leq \norm{\nabla[(\phi_\epsilon-\phi_0)w]}_{\infty}.\]
Since the potentials $\phi_\epsilon$ and $\phi_0$ are smooth enough, we may apply \eqref{eq:gn_particular_parameters} to $u:=(\phi_\epsilon-\phi_0)w$, so that
\begin{equation}
\label{eq:step_1}
\norm{\nabla[(\phi_\epsilon-\phi_0)w]}_{\infty} \leq C_{\mathrm{GN}} \norm{\nabla^2[(\phi_\epsilon-\phi_0)w]}_{\infty}^\frac{d+2}{d+4} \norm{(\phi_\epsilon-\phi_0)w}_2^{\frac{2}{d+4}}.
\end{equation}

\textbf{Step 2:} Now we will bound the first term on the right-hand side of \eqref{eq:step_1}:
\begin{equation}
\label{eq:step_2_goal}
\norm{\nabla^2[(\phi_\epsilon-\phi_0)w]}_{\infty}^\frac{d+2}{d+4}.
\end{equation}
Note that
\begin{align*}
\norm{\nabla^2[(\phi_\epsilon-\phi_0)w]}_{\infty} &\leq \norm{w\nabla^2(\phi_\epsilon-\phi_0)}_{\infty} + \norm{(\phi_\epsilon-\phi_0)\nabla^2 w}_{\infty} + 2 \norm{\nabla(\phi_\epsilon-\phi_0)\nabla w^\intercal}_{\infty}\\
&\leq \norm{\nabla^2(\phi_\epsilon-\phi_0)}_{\infty} + \norm{\phi_\epsilon-\phi_0}_{
\mathrm{B}(0,R),\infty} \times \norm{\nabla^2 w}_{\infty}\\
&\quad \quad+ 2 \norm{\nabla\phi_\epsilon-\nabla\phi_0}_{\mathrm{B}(0,R),\infty} \times \norm{\nabla w}_{\infty}.
\end{align*}
For the first term on the right-hand side, \ref{as:a1} plays a key role: note that for every $\epsilon > 0$, the upper bound in Theorem \ref{thm:chewi-pooladian} is such that
$$\frac12 \left(\sqrt{\frac{4\alpha}{\beta} + \epsilon^2 \alpha^2}-\epsilon \alpha\right) \leq \sqrt{\frac{\alpha}{\beta}}.$$
In particular, using also the estimate granted by Theorem \ref{thm:caffarelli}, we get
$$\norm{\nabla^2\phi_\epsilon - \nabla^2\phi_0}_\infty \leq \norm{\nabla^2\phi_\epsilon}_\infty + \norm{\nabla^2\phi_0}_\infty \leq 2 \sqrt{\frac{\alpha}{\beta}}.$$
In order to bound the two remaining terms on the right-hand side, let us state and prove the following auxiliary lemma.

\begin{lem}
\label{lem:correction}
In the above context, there exists $M = M(R, \mu, \nu) > 0$ depending on $R$, $\mu$ and $\nu$ such that for any $\epsilon >0$, there exists $x^*_\epsilon \in \mathrm{B}(0,R)$ satisfying both
\begin{equation}
\label{eq:lem_correction_1}
\abs{\nabla\phi_\epsilon(x^*_\epsilon)-\nabla\phi_0(x^*_\epsilon)} < M
\end{equation}
and
\begin{equation}
\label{eq:lem_correction_2}
\quad \abs{\phi_\epsilon(x^*_\epsilon)-\phi_0(x^*_\epsilon)} < M.
\end{equation}
\end{lem}

\begin{proof}
To start, let us prove that
$$\sup_{\epsilon>0}\int_{\R^d} \abs{\nabla\phi_\epsilon-\nabla\phi_0}^2\, d\mu < +\infty,$$
so let $\epsilon >0$. First, recall equation \eqref{eq:gradient_eot_potential}:
\[\nabla\phi_\epsilon(x) = \int_{\R^d} y \, d\pi_{\epsilon}^x(y).\]
By Jensen's inequality and the fact that $\pi_\epsilon \in \Pi(\mu,\nu)$, we have that
\[\int_{\R^d} \abs{\nabla\phi_\epsilon(x)}^2 \, d\mu(x) \leq \int_{\R^d} \int_{\R^d} \abs{y}^2 d\pi_{\epsilon}^x(y) \, d\mu(x) = \int_{\R^d\times \R^d} \abs{y}^2 \, d\pi_\epsilon(x,y) = \int_{\R^d} \abs{y}^2 d\nu(y).\]
On the other hand, note that
\[\int_{\R^d} \abs{\nabla\phi_0(x)}^2 \, d\mu(x) = \int_{\R^d} \abs{y}^2 \, d\nu(y)\]
since the Brenier map $x \mapsto \nabla \phi_0(x)$ pushes forward $\mu$ towards $\nu$. Then
\begin{align*}
\int_{\R^d} \abs{\nabla\phi_\epsilon-\nabla\phi_0}^2\, d\mu &\leq 2 \int_{\R^d} \abs{\nabla\phi_\epsilon}^2\, d\mu + 2 \int_{\R^d} \abs{\nabla\phi_0}^2\, d\mu \leq 4 \int_{\R^d} \abs{y}^2 \, d\nu(y),
\end{align*}
so $\sup_{\epsilon>0}\int_{\R^d} \abs{\nabla\phi_\epsilon-\nabla\phi_0}^2\, d\mu < +\infty$ since $\nu$ has finite moments of all orders; recall Remarks \ref{rk:finite_moments} and \ref{rk:also_nu}.

Let $M_1 > 0$ to be fixed. By Chebyshev's inequality, we have
\[\mu(\{x \in \R^d : \abs{\nabla\phi_\epsilon(x)-\nabla\phi_0(x)} \geq M_1\}) \leq \frac{1}{M^2_1}\int_{\R^d} \abs{\nabla\phi_\epsilon-\nabla\phi_0}^2 \, d\mu \leq \frac{4}{M^2_1} \int_{\R^d} \abs{y}^2 \, d\nu(y),\]
so we may fix $M_1$ big enough such that $\frac{4}{M^2_1}\int_{\R^d} \abs{y}^2 \, d \nu(y) \leq \frac15$. Hence,
\begin{equation}
\label{eq:lem_correction_proof_1_aux_aux_aux}
\mu(\{x\in \R^d : \abs{\nabla\phi_\epsilon(x)-\nabla\phi_0(x)} < M_1\}) > \frac45,
\end{equation}
so
\begin{equation}
\label{eq:lem_correction_proof_1}
\mu(\{x\in \mathrm{B}(0,R): \abs{\nabla\phi_\epsilon(x)-\nabla\phi_0(x)} < M_1\}) > \frac35.
\end{equation}
Without loss of generality, we may assume that the normalization \eqref{eq:norms_pots} holds so that \eqref{eq:extra_poincare} is verified. Similarly to the previous argument, we deduce that there exists $M_2 > 0$ such that
\begin{equation}
\label{eq:lem_correction_proof_2}
\mu(\{x\in \mathrm{B}(0,R) : \abs{\phi_\epsilon(x)-\phi_0(x)} < M_2\}) > \frac35.
\end{equation}
Let $M  :=  \max\{M_1,M_2\}$, which depends on $R$, $\mu$, and $\nu$. If we put together \eqref{eq:lem_correction_proof_1} and \eqref{eq:lem_correction_proof_2}, we deduce that
\[\mu(\{x\in \mathrm{B}(0,R) : \abs{\phi_\epsilon(x)-\phi_0(x)} < M, \abs{\nabla\phi_\epsilon(x)-\nabla\phi_0(x)} < M\}) > 0,\]
which yields the existence of $x^*_\epsilon \in \mathrm{B}(0,R)$ such that both \eqref{eq:lem_correction_1} and \eqref{eq:lem_correction_2} hold.
\end{proof}

Recall that we wanted to bound
\begin{equation}
\label{eq:aux_new_bound_reorg_1}
\norm{\nabla\phi_\epsilon-\nabla\phi_0}_{\mathrm{B}(0,R),\infty}
\end{equation}
and
\begin{equation}
\label{eq:aux_new_bound_reorg_2}
\norm{\phi_\epsilon-\phi_0}_{
\mathrm{B}(0,R),\infty}.
\end{equation}
By Lemma \ref{lem:correction}, there exists $M = M(R,\mu, \nu)> 0$ and $x^*_\epsilon \in \mathrm{B}(0,R)$ verifying both \eqref{eq:lem_correction_1} and \eqref{eq:lem_correction_2}. For \eqref{eq:aux_new_bound_reorg_1}, note that
\begin{align*}
\sup_{x \in \mathrm{B}(0,R)}\abs{\nabla\phi_\epsilon(x)-\nabla\phi_0(x)} &\leq \sup_{x \in \mathrm{B}(0,R)}\abs{\nabla\phi_\epsilon(x)-\nabla\phi_\epsilon(x^*_\epsilon)} + \sup_{x \in \mathrm{B}(0,R)}\abs{\nabla\phi_0(x)-\nabla\phi_0(x^*_\epsilon)}\\
&\quad \quad + \abs{\nabla\phi_\epsilon(x^*_\epsilon)-\nabla\phi_0(x^*_\epsilon)}\\
& \leq 4\sqrt{\frac{\alpha}{\beta}} R + M,
\end{align*}
where we used the second-order estimates holding under \ref{as:a1} and \eqref{eq:lem_correction_1}. For the remaining term \eqref{eq:aux_new_bound_reorg_2}, observe that a Taylor expansion combined with the aforementioned second-order estimates yields
\begin{align*}
\sup_{x \in \mathrm{B}(0,R)}\abs{\phi_\epsilon(x)-\phi_0(x)} &\leq \abs{\phi_\epsilon(x^*_\epsilon)-\phi_0(x^*_\epsilon)} + 2R\abs{\nabla\phi_\epsilon(x^*_\epsilon)-\nabla\phi_0(x^*_\epsilon)} + 4R^2\sqrt{\frac{\alpha}{\beta}}\\
&\leq M + 2RM + 4R^2\sqrt{\frac{\alpha}{\beta}}.
\end{align*}
In conclusion, if we combine the three bounds, we know that there exists a constant $C = C(K,\mu,\nu) > 0$ depending on $K$ (since the constant depends on $R$, which in turn depends on $K$), $\mu$, and $\nu$ such that for every $\epsilon > 0$,
\begin{equation}
\label{eq:step_2}
\norm{\nabla^2[(\phi_\epsilon-\phi_0)w]}_{\infty}^\frac{d+2}{d+4} \leq C.
\end{equation}

\textbf{Step 3:} Now we aim to bound the remaining term in the right-hand side of \eqref{eq:step_1} in an asymptotic way in order to obtain a quantity that converges to $0$ as $\epsilon$ vanishes; that is, we want to bound
\[\norm{(\phi_\epsilon-\phi_0)w}_2^{\frac{2}{d+4}}.\]
To control this term, the bound \eqref{eq:step_3_1} will be crucial. Indeed, since $\supp(w) \subseteq \overline{\mathrm{B}}(0,R)$, then we have that
\begin{align*}
\norm{(\phi_\epsilon-\phi_0)w}_2^2 &= \int_{\overline{\mathrm{B}}(0,R)}\abs{\phi_\epsilon(x)-\phi_0(x)}^2 w(x)\,dx = \int_{\overline{\mathrm{B}}(0,R)}\abs{\phi_\epsilon(x)-\phi_0(x)}^2 e^{V(x)} w(x) e^{-V(x)}\, dx\\
&\leq \int_{\overline{\mathrm{B}}(0,R)}\abs{\phi_\epsilon(x)-\phi_0(x)}^2 e^{V(x)} \, d\mu(x)\\
&\leq \norm{e^V}_{\overline{\mathrm{B}}(0,R),\infty} \norm{\phi_\epsilon-\phi_0}_{L^2(\mu)}^2,
\end{align*}
so if we combine this with \eqref{eq:step_3_1}, we get
\begin{equation}
\label{eq:step_3}
\norm{(\phi_\epsilon-\phi_0)w}_2^{\frac{2}{d+4}} \leq C' \epsilon^{\frac{1}{d+4}},
\end{equation}
where $C' = C(d,K,\mu, \nu)>0$ is a constant depending on $d, K$ (since it depends on $R$ and the cutoff function $w$), $\mu$ and $\nu$. Finally, if we put together \eqref{eq:step_1}, \eqref{eq:step_2}, and \eqref{eq:step_3}, we deduce the existence of a constant $C_{\mathrm{grad}} = C_{\mathrm{grad}}(d,K,\mu,\nu) >0$ depending on the dimension $d$, the compact set $K$, and the measures $\mu$ and $\nu$ such that, uniformly on $\epsilon >0$,
$$\norm{\nabla\phi_\epsilon - \nabla\phi_0}_{K,\infty} \leq C_{\mathrm{grad}}\, \epsilon^{\frac{1}{d+4}}.$$

\end{proof}

We end by proving Theorem \ref{thm:main_theorem_potentials}.

\begin{proof}[Proof (of Theorem \ref{thm:main_theorem_potentials}):]
Let $K \subseteq \R^d$ be a compact and connected set. Since we have assumed that the potentials are normalized as $\int_{\R^d}\phi_\epsilon\,d\mu = \int_{\R^d}\phi_\epsilon\,d\mu = 0$, we can use \cite[Theorem 1, p. 290]{evans10} and the bound \eqref{eq:step_3} to get
$$\norm{\phi_\epsilon - \phi_0}_{K,\infty} \leq C'' \left(\norm{\nabla\phi_\epsilon - \nabla\phi_0}_{K,\infty} + \epsilon\right),$$
where $C''=C''(d,K,\mu,\nu)> 0$ is a constant depending on $d, K, \mu$ and $\nu$. The conclusion follows from the estimate given by Theorem \ref{thm:main_theorem_gradients}.
\end{proof}

\section{The Gaussian case}
\label{sec:section4}
In this section, we provide the proof of Proposition \ref{prop:gaussian_case}.
\begin{proof}[Proof (of Proposition \ref{prop:gaussian_case}):]
For Gaussian measures, we may compute explicitly $\nabla\phi_0$ and $\nabla\phi_\epsilon$ (see \cite{gelbrich90,janati2020entropic,del2020statistical,mallasto22}): both are linear maps given by
\begin{equation*}
\label{eq:gaussian_ot}
\forall x \in \R^d, \quad \nabla\phi_0(x) = A^{-\frac12}(A^{\frac12}BA^{\frac12})^{\frac12}A^{-\frac12} x
\end{equation*}
and
\begin{equation*}
\label{eq:gaussian_eot}
\forall x \in \R^d, \quad \nabla\phi_\epsilon(x) = \left(A^{-\frac12}\left(A^{\frac12}BA^{\frac12} + \frac{\epsilon^2}{4}I_d\right)^{\frac12}A^{-\frac12} - \frac\epsilon2A^{-1}\right)x.
\end{equation*}
Now fix $R > 0$ and let $K= \overline{\mathrm{B}}(0,R) \subseteq \R^d$, so that
$$\sup_{x \in K}\abs{\nabla\phi_\epsilon(x) - \nabla\phi_0(x)} = R \abs{\left(A^{-\frac12}\left(A^{\frac12}BA^{\frac12} + \frac{\epsilon^2}{4}I_d\right)^{\frac12}A^{-\frac12} - \frac\epsilon2A^{-1}\right) - A^{-\frac12}(A^{\frac12}BA^{\frac12})^{\frac12}A^{-\frac12}}_{\mathrm{op}},$$
where for a matrix $M \in \mathcal{M}_d(\R)$ we write $\abs{M}_\mathrm{op} := \sup_{\abs{x}=1}\abs{Mx}$ for its operator norm.

Let us note that we may expand the matrix $\left(A^{\frac12}BA^{\frac12} + \frac{\epsilon^2}{4}I_d\right)^{\frac12}$, using a Taylor expansion of order one for the matrix square root function around the point $A^{\frac12}BA^{\frac12}$ (see \cite[Theorem 1.1]{delmoral18}):
$$\left(A^{\frac12}BA^{\frac12} + \frac{\epsilon^2}{4}I_d\right)^{\frac12} = (A^{\frac12}BA^{\frac12})^{\frac12} + \frac{\epsilon^2}{4} \int_0^{+\infty} e^{-2t(A^{\frac12}BA^{\frac12})^{\frac12}} \, dt + R(A,B,\epsilon),$$
where $R(A,B,\epsilon)$ is a matrix such that
$$\abs{R(A,B,\epsilon)}_{\mathrm{op}} \leq \frac{\epsilon^4}{32} \lambda_{\mathrm{min}}(A^{\frac12}BA^{\frac12})^{-\frac32} = \frac{\epsilon^4}{32} \abs{(A^{\frac12}BA^{\frac12})^{-1}}_\mathrm{op}^\frac{3}{2},$$
where $\lambda_{\mathrm{min}}(\cdot)$ denotes the smallest eigenvalue of a positive-definite matrix. On the other hand, note that
$$\abs{\int_0^{+\infty} e^{-2t(A^{\frac12}BA^{\frac12})^{\frac12}} \, dt}_\mathrm{op} \leq \int_0^{+\infty} e^{-2t\lambda_{\mathrm{min}}(A^{\frac12}BA^{\frac12})^{\frac12}}\,dt = \frac{1}{2\lambda_{\mathrm{min}}(A^{\frac12}BA^{\frac12})^\frac{1}{2}} = \frac12 \abs{(A^{\frac12}BA^{\frac12})^{-1}}_\mathrm{op}^\frac{1}{2}.$$

In consequence, we have that
\begin{align*}
    \sup_{x \in K}\abs{\nabla\phi_\epsilon(x) - \nabla\phi_0(x)} &= R \abs{\left(A^{-\frac12}\left(A^{\frac12}BA^{\frac12} + \frac{\epsilon^2}{4}I_d\right)^{\frac12}A^{-\frac12} - \frac\epsilon2A^{-1}\right) - A^{-\frac12}(A^{\frac12}BA^{\frac12})^{\frac12}A^{-\frac12}}_{\mathrm{op}}\\
    &\leq R \abs{A^{-\frac12}\left(A^{\frac12}BA^{\frac12} + \frac{\epsilon^2}{4}I_d\right)^{\frac12}A^{-\frac12} - A^{-\frac12}(A^{\frac12}BA^{\frac12})^{\frac12}A^{-\frac12}}_{\mathrm{op}} + R \abs{\frac{\epsilon}{2} A^{-1}}_{\mathrm{op}}\\
    &\leq \epsilon R \frac{\abs{A^{-1}}_{\mathrm{op}}}2 \left(\epsilon \frac{\abs{(A^{\frac12}BA^{\frac12})^{-1}}_\mathrm{op}^\frac{1}{2}}{4} + \epsilon^3\frac{\abs{(A^{\frac12}BA^{\frac12})^{-1}}_\mathrm{op}^\frac{3}{2}}{16} + 1\right),
\end{align*}
so $\sup_{x \in K}\abs{\nabla\phi_\epsilon(x) - \nabla\phi_0(x)} = O(\epsilon)$ as $\epsilon \to 0$, with a constant which depends only on $R$ and the matrices $A$ and $B$.

\end{proof}

\bibliographystyle{alpha}
\bibliography{main.bib}

\end{document}